\documentclass[12pt]{amsart}

\usepackage{amsfonts,amsmath,amssymb}

\usepackage{enumerate}

\numberwithin{equation}{section}

\newtheorem{theorem}{Theorem}[section]

\newtheorem{lemma}[theorem]{Lemma}

\newtheorem{example}[theorem]{Example}

\newtheorem{definition}{Definition}[section]

\newtheorem{corollary}[theorem]{Corollary}

\newtheorem{remark}[theorem]{Remark}

\newtheorem{remarks}[theorem]{Remarks}

\newcommand{\cl}[1]{\mathcal{#1}} 

\newcommand{\bb}[1]{\mathbb{#1}}

\newcommand{\sca}[1]{\left\langle#1\right\rangle} 

\newcommand{\nor}[1]{\left\Vert #1\right\Vert}

\newcommand{\Lat}[1]{\mathrm{Lat}(#1)} 

\newcommand{\Alg}[1]{\mathrm{Alg}(#1)} 

\begin{document}

\title{On stable maps of operator algebras }

\author[G.K. Eleftherakis ]{G.K. Eleftherakis }

\address{G. K. Eleftherakis\\ University of Patras\\Faculty of Sciences\\ Department of Mathematics\\265 00 Patras Greece }
\email{gelefth@math.upatras.gr}





\date{}

\maketitle

\begin{abstract} We define a strong Morita-type equivalence $\sim _{\sigma \Delta }$ for operator algebras. 
We prove that $A\sim _{\sigma \Delta }B$ if and only if $A$ and $B$ are stably isomorphic. We 
also define a relation $\subset _{\sigma \Delta }$ for operator algebras. We prove 
that if $A$ and $B$ are $C^*$-algebras, then $A\subset _{\sigma \Delta } B$ if and only if 
there exists an onto $*$-homomorphism $\theta :B\otimes \cl K  \rightarrow A\otimes \cl K,$ 
where $\cl K$ is the set of compact operators acting on an infinite
 dimensional separable Hilbert space. Furthermore, we prove that if $A$ and $B$ are $C^*$-algebras such that   
$A\subset _{\sigma \Delta } B$ and $B\subset _{\sigma \Delta } A $,  then there exist projections $r, \hat r$ in the centers 
of $A^{**}$ and $B^{**}$, respectively, such that 
$Ar\sim _{\sigma \Delta }B\hat r$ and $A (id_{A^{**}}-r)
 \sim _{\sigma \Delta }B(id_{B^{**}}-\hat r). $ \end{abstract}

\section{Introduction}

Two operator algebras $A$ and $B$ are called stably isomorphic if the algebras $A\otimes \cl K$ and $B\otimes \cl K$ 
are isomorphic as operator algebras. Here, $\cl K$ is the set of compact operators acting on $l^2(\bb N).$ Stably isomorphic $C^*$-algebras 
are strongly Morita equivalent in the sense of Rieffel. The same is true of non-self-adjoint operator algebras if we 
consider the strong Morita equivalence that was introduced by Blecher, Muhly, and Paulsen in \cite{bmp}. Meanwhile, the 
converse is not true, even in the case of $C^*$-algebras \cite{bgr}. 

We introduce a new Morita type equivalence between operator algebras: Let $A$ and $B$ be operator algebras that are possibly non-self-adjoint. We say that   $A$ and $B$ are $\sigma-$strongly $ \Delta $-equivalent, and write $A\sim _{\sigma \Delta }B$, if 
there exist completely isometric homomorphisms $\alpha: A\rightarrow  \alpha(A),\;\; \beta: B\rightarrow  \beta (B)$ 
and a $\sigma $-ternary ring of operators $M$ such that 
\begin{equation}\label{giasemi}\alpha (A)=\overline{[M^*\beta (B)M]}^{\|\cdot\|} , 
\;\; \beta (B)=\overline{[M\alpha (A)M^*] }^{\|\cdot\|}. \end{equation}
See the definition of the $\sigma -$ternary ring of operators in Definition \ref{110}.

In the proof of \cite[Theorem 3.2]{elehoust}, see also  \cite[Lemma 3.4]{elehoust}, we noticed that if $A, B$ are operator algebras 
possessing countable approximate identities, $M$ is a ternary ring of operators and the triple $(A, B, M)$ satisfies (\ref{giasemi}) 
 then $M$ is necessarily a $\sigma -$ternary ring of operators. We used this fact in order to prove that $A$ and 
$B$ are stably isomorphic. Subsequently in \cite[Theorem 4.6]{elekak} we extended the proof in the case of operator spaces.
In the present paper, we prove that $\sim  _{\sigma \Delta }$ is an equivalence relation in the class of operator algebras 
and we use this fact to prove that 
$A\sim _{\sigma \Delta }B$ if and only if $A$ and $B$ are stably isomorphic. 

In \cite{elehoust}, we studied the relationship between 
$A$ and $B$ when (\ref{giasemi}) holds for a ternary ring (TRO) of operators $M$ that is not necessarily a
$\sigma $-TRO. This relation is not equivalent to the existence of an operator algebra isomorphism between 
$A\otimes \cl K$ and $B\otimes \cl K.$ 

We also consider a weaker relation $\subset _{\sigma \Delta }$ between operator algebras:  
We say that $A ,$  $\sigma \Delta $-embeds into $B$ 
if there exists a projection $p$ in the center of $\Delta (B^{**}),$ where  $\Delta (B^{**})$ is the diagonal 
of the second dual operator algebra of $B,$ such that $pBp$ is an operator algebra and $A\sim _{\sigma \Delta }pBp.$ 
In this case, we write $A\subset _{\sigma \Delta }B.$ We prove that $\subset _{\sigma \Delta }$ 
is transitive. For the case of $C^*$-algebras, we prove that $A\subset _{\sigma \Delta }B$ 
if and only if there exists an onto $*$-homomorphism from  $B\otimes \cl K$ onto $A\otimes \cl K,$ 
which is true if and only if there exists an ideal $I$ of $B$ such that $A\sim _{\sigma \Delta }B/I.$ 

We investigate whether it is true that $A\sim _{\sigma \Delta }B$ if $A\subset _{\sigma \Delta }B$ and  $B\subset _{\sigma \Delta }A.$ 
In general, this is not true (see Section \ref{non}). It is also not true even in the case of $C^*$-
algebras (see Example \ref{2000000}). However, we prove that if $A$ and $B$ are $C^*$-algebras such that   
$A\subset _{\sigma \Delta } B$ and $B\subset _{\sigma \Delta } A $,  then there exist projections $r, \hat r$ in the centers 
of $A^{**}$ and $B^{**}$, respectively, such that 
$Ar\sim _{\sigma \Delta }B\hat r$ and $A (id_{A^{**}}-r)
 \sim _{\sigma \Delta }B(id_{B^{**}}-\hat r). $ A dual version of the results obtained in this article can be found in 
\cite{elest}.

In the following we describe the notations and symbols used in this paper. If $H, K$ are Hilbert spaces, then $B(H, K)$ is 
the space of bounded operators from  $H$ to $K.$ We write $B(H)$ for $B(H, H).$ A ternary ring of operators (TRO) 
is a subspace of some $B(H, K)$ satisfying $MM^*M\subseteq M$ (see the definition of a $\sigma $-TRO in Definition \ref{110}). 
An operator algebra is an operator space and Banach algebra for which there exists a completely isometric 
homomorphism $\alpha: A \rightarrow B(H).$ In this article, when we consider an operator algebra, 
we mean an operator algebra with a contractive approximate identity. We note that $C^*$-algebras possess contractive approximate identities
 automatically. If $X$ is an operator space, then $M_\infty (X)$ is the set of $\infty \times \infty $ 
matrices whose finite submatrices have uniformly bounded norm. The space   $M_\infty (X)$ 
is an operator space. In addition, $M_\infty ^{fin}(X)$ will denote the subspace of $M_\infty (X)$ consisting 
of "finitely supported matrices." We write $K_\infty (X)$ for the norm closure in $M_\infty (X)$ of $M_\infty ^{fin}(X).$ 
It is well-known that the space  $K_\infty (X)$ is completely isometric isomorphic with $X\otimes \cl K,$ 
where $\otimes $ is the minimal tensor product \cite{bm}. For further details on the operator space theory that is used in this paper, we refer the reader to the books by
\cite{bm}, \cite{er}, \cite{paul}, and \cite{pis}.

A nest $\cl N\subseteq B(H)$ is a totally ordered set of orthogonal projections containing the zero and identity operators that are closed under arbitrary suprema and infima. Given a nest   $\cl N\subseteq B(H),$ by $\Alg{ \cl N}$ 
we denote the corresponding nest algebra:
$$\{x\in B(H): (I_H-n)xn=0,\;\;\forall \;\;n\;\in \;\cl N\}.$$ 
Given an operator algebra $A,$ we denote its center by $Z(A)$ and its diagonal $A\cap A^*$ by $\Delta (A).$    
If $S$ is a subset of a vector space, then we denote the linear span of the elements of $S$ by $[S]$.

\section{Preliminaries}

The purpose of this section is to prove Lemma \ref{150}, which is required to prove that 
$\sim _{\sigma \Delta }$  is an equivalence relation in Section \ref{xxxxxx}.

\begin{definition}\label{110} Let $H,K$ be Hilbert spaces, and $M\subseteq B(H,K)$ be a norm closed TRO. We call $M$ $\sigma $-TRO 
if there exist sequences $\{m_i,\; n_i,\; i\in \bb N\}\subseteq M$ such that 
$$\lim_l\sum_{i=1}^lm_im_i^*m=m, \;\;\lim_l\sum_{i=1}^lmn_i^*n_i=m, \;\;\forall m\in M$$ and $$ \nor{ 
\sum_{i=1}^l m_im_i^*}\leq 1 , \;\;\nor{
\sum_{i=1}^l n_i^*n_i}\leq 1, \;\forall\; l.$$
\end{definition}

\begin{remark} A norm-closed TRO $M$ is a $\sigma $- TRO if and only if the $C^*$ algebras $\overline{[M^*M]}^{\|\cdot\|}, \overline{[MM^*]}^{\|\cdot\|} $ are $\sigma$-unital. A proof of this fact can be found in 
Theorem 2.1 in \cite{brown}.
\end{remark}

\begin{lemma}\label{120} Let $A\subseteq B(H), B\subseteq B(K)$ be $C^*$ algebras and $M\subseteq B(H,K)$ be a $\sigma $-TRO such that $$B=\overline{[M^*AM]}^ {\|\cdot\|}, \;\;\; MBM^*\subseteq A.$$ If $A$ is $\sigma $-unital, then $B$ is $\sigma $-unital.
\end{lemma}
\begin{proof} 
Suppose that $(a_n)_{n\in \bb N}\subseteq A$ such that $\lim_n a_na=a\; \forall \;\;a\in\; A.$  
In addition, let $\{m_i: i\in \bb N\}\subseteq M$ be such that $\lim_l\sum _{i=1}^lm_i^*m_im^*=m^* \forall\; m\;\in M$  and
 $\nor{ \sum_{i=1}^l m_i^*m_i}\leq 1 \;\forall\; l.$ It suffices to prove 
that $B$ contains a strictly positive element. Define $$b= \sum_{l=1}^\infty \sum_{n=1}^\infty  \sum_{i,j=1}^l\frac{m_i^*a_nm_im_j^*a_n^*m_j}{
2^n2^l\|a_n\|^2}.$$ Because $$\nor{ \sum_{i,j=1}^l m_i^*a_nm_im_j^*a_n^*m_j  }=\nor{\sum_{i=1}^lm_i^*a_n^*m_i}^2\leq 
\|a_n\|^2,$$ we have that $$\sum_{l=1}^\infty \sum_{n=1}^\infty  \nor{\sum_{i,j=1}^l\frac{m_i^*a_nm_im_j^*a_n^*m_j }{2^n2^l\|a_n\|^2}}<+\infty .$$
Thus, the element $b$ is well defined. Observe that $b\geq 0$ if $\phi $ is a state of $B$ such that  
$$\phi (b)=0\Rightarrow \phi (\sum_{i,j=1}^l m_i^*a_nm_im_j^*a_n^*m_j  ) =0,\;\;\forall \;\; n,\;l.$$
If $a\in A, m,s,t,r,\in M,$ then the Cauchy-Schwartz inequality for $\phi $   implies that 
$$\phi (\sum_{i=1}^lm_i^*a_nm_im^*ast^*r)=0\;\; \forall\;\; n,\;l.$$
Because $$m_im^*ast^*r\in MM^*AMM^*M\subseteq AM,$$ we have that $$\lim_n a_nm_im^*st^*r=m_im^*st^*r.$$
Thus, $\phi (\sum_{i=1}^lm_i^*m_im^*astr^*)=0 ,\; \forall\; l.$ Because 
$$\lim_l\sum_{i=1}^lm_i^*m_im^*=m^*,$$ we have that $\phi (m^*ast^*r)=0$ for all $m,s, t,r\in M, a\in A.$ 
Because $B=\overline{[M^*AMM^*M]} ^{\|\cdot\|}$, we conclude that $\phi =0$. This contradiction shows that $b$ is strictly positive.

\end{proof}

\begin{lemma}\label{130} Let $E,F,M_1, M_2$ be TROs such that the algebra $\overline{[M_2^*M_2]} ^{\|\cdot\|}$ is $\sigma $ -
unital and 
$$E=\overline{[M_2^*FM_1]} ^{\|\cdot\|},  \;\;\; F=\overline{[M_2EM_1^*]} ^{\|\cdot\|}.$$ 
If it also holds that the algebra  $ \overline{[EE^*]} ^{\|\cdot\|} $ is $\sigma $-unital, then the algebra $\overline{[FF^*]} ^{\|\cdot\|}
$  is also $\sigma $-unital. 
\end{lemma}
\begin{proof} Observe that $$ \overline{[FF^*]} ^{\|\cdot\|} =\overline{[M_2EE^*M_2^*]} ^{\|\cdot\|} .$$  
 Let $\{m_i: i\in \bb N\}\subseteq M_2$ be such 
that $\sum_{i=1}^lm_im_i^*m=m \forall\; m\;\in\; M_2,$$$
\nor{\sum_{i=1}^lm_im_i^*}\leq 1\;\forall \;l,$$ and let $(a_n)_n\subseteq \overline{[EE^*]} ^{\|\cdot\|} $ be a $\sigma $-unit.
As in Lemma \ref{120}, we can prove that the element $$b= 
\sum_{l=1}^\infty \sum_{n=1}^\infty  \sum_{i,j=1}^l\frac{m_ia_nm_i^*m_ja_n^*m_j^*}{
2^n2^l\|a_n\|^2}$$ is strictly positive in  $\overline{[FF^*]} ^{\|\cdot\|}.$ Thus,   $\overline{[FF^*]} ^{\|\cdot\|}$  is  $\sigma $-unital. 
\end{proof}

\begin{lemma}\label{140} Let $E, F, M_1, M_2$ be TROs such that $M_1, M_2, F$ are $\sigma $-TROs and $$E=
\overline{[M_2FM_1^*]} ^{\|\cdot\|}, \;\; M_2^*EM_1\subseteq F.$$ Then, $E$ is a $\sigma $-TRO.
\end{lemma}
\begin{proof}
It suffices to prove that the $C^*$-algebras $ \overline{[EE^*]} ^{\|\cdot\|},  \overline{[E^*E]} ^{\|\cdot\|} $   are $\sigma $-unital. Define the $C^*$-algebras 
$$ \Pi (E)=\left(\begin {array}{clr} \overline{[E^*E]} ^{\|\cdot\|} & E^*\\ E&  \overline{[EE^*]} ^{\|\cdot\|} \end {array}\right), \;\;\;
\Pi (F)=\left(\begin {array}{clr} \overline{[F^*F]} ^{\|\cdot\|} & F^*\\ F&  \overline{[FF^*]} ^{\|\cdot\|} \end {array}\right). $$ 
Because $F$ is a $\sigma $-TRO, the algebra $\Pi (F)$ is $\sigma $-unital. Furthermore, 
it easy to see that $$ (M_1\oplus M_2)\Pi (F)(M_1\oplus M_2)^*=\Pi (E) $$ and
 $$(M_1\oplus M_2)^*\Pi (E)(M_1\oplus M_2)^*\subseteq \Pi (F) .$$ Lemma \ref{120} implies that $\Pi (E)$ is $\sigma $-unital. Thus, 
 the $C^*$-algebras $ \overline{[EE^*]} ^{\|\cdot\|},  \overline{[E^*E]} ^{\|\cdot\|} $   are $\sigma $-unital. 
\end{proof}

\begin{lemma}\label{150} Let $H, K,  L$ be Hilbert spaces, $M\subseteq B(H,K), N\subseteq B(K, L)$ be $\sigma $-TROs, and $D$ be the $C^*$ algebra generated by the sets $MM^*, N^*N.$ Then, $T=\overline{[NDM]} ^{\|\cdot\|}$ 
is a $\sigma $-TRO.
\end{lemma}
\begin{proof} We have that $$NDMM^*DN^*NDM\subseteq NDM.$$ Thus, $TT^*T\subseteq T$, and so $T$ is a TRO. We define the TRO 
$$Z=\left(\begin{array}{clr} \overline{[M^*D]}  ^{\|\cdot\|}\\ \overline{[ND]} ^{\|\cdot\|}\end{array}\right).$$ Then, 
$$ZZ^*=\left(\begin{array}{clr} M^*DM & M^*DN^* \\ NDM & NDN^*\end{array}\right).$$ 
Let $$\{m_i: i\in \bb N\}\subseteq M,\;\; \{n_i: i\in \bb N\}\subseteq N$$ be such that
 $$\nor{\sum _{i=1}^lm_i^*m_i}\leq 1, \;\; \nor{\sum _{i=1}^ln_in_i^*}\leq 1, \forall l$$ and  
$$ \lim_l \sum _{i=1}^lm_i^*m_i m^*=m^*  \;\forall\; m\;\in M,\;\; \lim_l\sum _{i=1}^ln_in_i^*n=n\; \forall\; n\;\in \;N. $$ 
The elements $$a_l= \left(\begin{array}{clr} \sum _{i=1}^lm_i^*m_i  & 0 \\ 0 & \sum _{i=1}^ln_in_i^*\end{array} \right), l\in \bb N$$  
 belong to $$\left(\begin{array}{clr} M^*M  & 0 \\ 0 &  NN^* \end{array} \right)\subseteq ZZ^*, $$ and satisfy 
$\lim_l a_lx=x, \forall x\in \overline{[ZZ^*]} ^{\|\cdot\|}.$ Thus, $ \overline{[ZZ^*]} ^{\|\cdot\|}  $ is a $\sigma $-unital $C^*$- algebra. 
Now, we have that $Z=\overline{[ZD]} ^{\|\cdot\|}$ and $$\overline{[Z^*Z]} ^{\|\cdot\|}= \overline{DMM^*D+DNN^*D} ^{\|\cdot\|}.$$ We can easily see that 
$D=\overline{[MM^*D+N^*ND]} ^{\|\cdot\|}$, and thus $$\overline{[Z^*Z]} ^{\|\cdot\|}  =\overline{DD} ^{\|\cdot\|}=D.$$ 
Now, apply Lemma \ref{130} for $$M_1=\bb C,\; M_2=Z^*,\; E=Z,\; F=D. $$ We obtain that 
$$ \overline{[M_2EM_1^*]} ^{\|\cdot\|}=\overline{[Z^*Z]} ^{\|\cdot\|}=D=F,\;\;\; \overline{[M_2^*FM_1]} ^{\|\cdot\|}=\overline{[ZD]} ^{\|\cdot\|}=Z=E, $$

$\overline{[M^*_2M_2]} ^{\|\cdot\|}=\overline{[ZZ^*]} ^{\|\cdot\|}=\overline{[EE^*]} ^{\|\cdot\|}$ is $\sigma $-unital. Lemma \ref{130} implies that 
$\overline{[FF^*]} ^{\|\cdot\|}$ is 
$\sigma $-unital, and thus $D$ is $\sigma $-unital $C^* $-algebra. Now,
$$\overline{[NDM]} ^{\|\cdot\|}=T, \;\;N^*TM^*=N^*NDMM^*\subseteq D.$$ Lemma \ref{140} implies that $T$ is a $\sigma $-TRO.
\end{proof}

\section{   $\sigma $-strong $\Delta $-equivalence}\label{xxxxxx}

\begin{definition} \label{221}Let $A$ and $B$ be operator algebras acting on the Hilbert spaces $H$ and $L$, respectively. We call them 
$\sigma $-strongly $TRO $-equivalent if there exists a $\sigma $-TRO $M\subseteq B(L, H)$ such that 
$$A=\overline{[M^*BM]}^{\|\cdot\|} , \;\; B=\overline{[MAM^*] }^{\|\cdot\|}. $$ In this case, we write $A\sim _{\sigma TRO}B.$ 
\end{definition}

\begin{definition} \label{222}Let $A$ and $B$ be operator algebras. We call these 
$\sigma $-strongly $\Delta  $-equivalent if there exist completely isometric homomorphisms $\alpha: A\rightarrow  \alpha(A), \; \beta:
B\rightarrow  \beta (B)$ such that $\alpha (A)\sim _{\sigma TRO}\beta (B).$ 
In this case, we write $A\sim _{\sigma \Delta }B.$ 
\end{definition}

\begin{theorem}\label{223} Let $A, B$ be $\sigma $-strongly $\Delta  $-equivalent operator algebras. Then, for every completely 
isometric homomorphism  $\alpha: A\rightarrow  \alpha(A)$ there exists a completely isometric homomorphism $\beta:
B\rightarrow  \beta (B)$ such that $\alpha (A)\sim _{\sigma TRO}\beta (B).$ 
\end{theorem}
\begin{proof} We may assume that $H, L, M$ are as in Definition \ref{221}. By $Y$, we denote the space $Y=\overline{[BMA]}
^{\|\cdot\|}.$ Let $K$ be the $A$-balanced Haagerup tensor product 
$K=Y\otimes ^h_AH.$ This is a Hilbert space \cite{bmp}. Define 
$$\beta : B\rightarrow B(K), \;\;\;\beta (b)(y\otimes h)=(by)\otimes h.$$ 
By Lemma 2.10 in \cite{elehoust}, $\beta $ is a completely isometric homomorphism. From the same article, if $m\in M,$ we define 
$$\mu (m): H\rightarrow K, \; \; \mu (m)(\alpha (a)(h))=(ma)\otimes h.$$ The map $\mu: M \rightarrow \mu (M)$ 
is a TRO homomorphism. Thus, $\mu (M)$ is a $\sigma $-TRO. By Theorem 2.12 in \cite{elehoust}, we have that
$$\alpha (A)=\overline{[\mu (M)^*\beta (B)\mu (M)]}^{\|\cdot\|} , \;\; \beta (B)=\overline{[\mu (M)\alpha (A)\mu (M)^*] }^{\|\cdot\|}. $$ 
The proof is complete.
\end{proof}

\begin{theorem}\label{224} The $\sigma $-strong $\Delta  $-equivalence of  operator algebras is an equivalence relation 
in the class of operator algebras. 
\end{theorem}
\begin{proof} It suffices to prove the transitivity property. Let $A, B,$ and $C$ be operator algebras such that 
$A\sim _{\sigma \Delta }B$ and $B\sim _{\sigma \Delta }C.$ Therefore, there exists a $\sigma $-TRO $M$ and completely isometric homomorphisms 
$\alpha: A\rightarrow  \alpha(A), \; \beta:
B\rightarrow  \beta (B)$ such that  $$\alpha (A)=\overline{[M^*\beta (B)M]}^{\|\cdot\|} , \;\; 
\beta (B)=\overline{[M\alpha (A)M^*] }^{\|\cdot\|}. $$ 
 By Theorem \ref{223}, there exists a $\sigma $-TRO $N$ and a completely isometric homomorphism $\gamma: C \rightarrow \gamma (C)$ 
such that   $$\beta (B)=\overline{[N^*\gamma (C)N]}^{\|\cdot\|} , \;\; 
\gamma (C)=\overline{[N\beta (B)N^*] }  ^{\|\cdot\|} . $$
Let $D$ be the $C^*$-algebra generated by the set $\{MM^*\}\cup \{N^*N\}.$ By Lemma \ref{150}, the 
space  $T=\overline{[NDM]} ^{\|\cdot\|}$ 
is a $\sigma $-TRO. As in the proof of Theorem 2.1 in \cite{elehoust}, we can prove that 
$$\alpha (A)=\overline{[T^*\gamma (C)T]} ^{\|\cdot\|} , \gamma (C)=\overline{[T\alpha (A)T^*]} ^{\|\cdot\|} .$$
Thus, $A\sim _{\sigma \Delta }C. $ 
\end{proof}

 \begin{theorem}\label{225} Let $A, B$ be operator algebras. Then, $A$ and $B$ are 
$\sigma $-strongly $\Delta  $-equivalent if and only if they are stably isomorphic.  
\end{theorem}
\begin{proof} We assume that  $M$  is a $\sigma $-TRO satisfying 
$$A=\overline{[M^*BM]}^{\|\cdot\|} , \;\; B=\overline{[MAM^*] }^{\|\cdot\|}. $$
Theorem 4.6 in \cite{elekak} implies that there exists a completely isometric onto linear map $K_\infty (A)\rightarrow K_\infty (B).$ 
By using the Banach-Stone theorem for operator algebras, we may assume that this map is also a homomorphism \cite[4.5.13]{bm}.

For the converse, suppose that $K_\infty (A)$ and $ K_\infty (B)$ are completely isometrically isomorphic  as operator algebras. Let $R_\infty $ 
be the space of infinite rows consisting of compact operators. Then, $R_\infty $ is a $\sigma $-TRO, and we have that
$$R_\infty K_\infty (A)R_\infty^*=A, \;\;\; \overline{[R_\infty ^*AR_\infty ]}^{\|\cdot\|}=K_\infty (A). $$
Thus, $A\sim _{\sigma TRO}K_\infty (A).$ Therefore, $A\sim _{\sigma \Delta }K_\infty (B).$
By the same arguments, $B\sim _{\sigma TRO}K_\infty (B).$ Therefore, Theorem \ref{224} implies that $A\sim _{\sigma \Delta }B.$   
 \end{proof}

\begin{corollary}\label{226} Rieffel's strong Morita equivalence of $C^*$-algebras is weaker than  
$\sigma $-strong $\Delta  $-equivalence. 
\end{corollary}
\begin{proof} It is well-known \cite{bgr} that there exist $C^*$-algebras that are 
strongly  Morita equivalent in the sense of Rieffel but are not stably isomorphic. Thus, 
by Theorem \ref{225} these  $C^*$-algebras cannot be $\sigma $-strongly $\Delta  $-equivalent.
\end{proof}

\begin{corollary}\label{227} Two $\sigma $-unital $C^*$-algebras are strongly  Morita equivalent in the sense of Rieffel 
if and only if they are $\sigma $-strongly  $\Delta  $-equivalent. 
\end{corollary}
\begin{proof} By \cite{bgr}, two  $\sigma $-unital $C^*$-algebras are strongly  Morita equivalent in the sense of Rieffel 
if and only if they are stably isomorphic. The conclusion is implied by Theorem \ref{225}.
\end{proof}

\section{Strong Morita embeddings}

In \cite{elest}, we defined a new relation $\subset _\Delta $ between dual operator algebras: Given two unital 
dual operator algebras $A$ and $B$, we say that $A\subset _\Delta B$ if there exists an orthogonal projection $p\in B$ 
such that $A$ and $pBp$ are weakly stably isomorphic. In this case, there exists a projection $q\in Z(\Delta (B))$ 
such that $pBp$ and $qBq$ are weakly stably isomorphic \cite[Lemma 2.11]{elest}. In the present section, we aim to investigate the strong version of the previously stated 
relation for operator algebras.

\begin{definition}\label{2100} Let $A$ and $B$ be operator algebras. We say that $A ,$  $\sigma \Delta $-embeds into $B,$ 
if there exists a projection $p\in Z(\Delta (B^{**}))$ such that $pBp$ is an operator algebra and $A\sim _{\sigma \Delta }pBp.$ 
In this case, we write $A\subset _{\sigma \Delta }B.$ 
\end{definition}

\begin{remark} Let $A$ be a $C^*$-algebra, and $p$ be a central projection of $A^{**}.$ Because the map 
$A\rightarrow A^{**}, \;a\rightarrow ap$ is a $*$-homomorphism, it has norm-closed range. Thus, $Ap$ is a $C^*$-algebra.
\end{remark}

In the following, we prove that $\subset _{\sigma \Delta }$ is transitive.

\begin{theorem}\label{2300} Let $A, B, C$ be operator algebras. If $A\subset _{\sigma \Delta }B$ 
and $B\subset _{\sigma \Delta }C$, then $A\subset _{\sigma \Delta }C.$
\end{theorem}
\begin{proof} Let $p\in Z(\Delta (B^{**})), q\in Z(\Delta (C^{**}))$ be such that $pBp, qCq$ are 
operator algebras and $ A\sim _{\sigma \Delta }pBp , B\sim _{\sigma \Delta }qCq.$ We write 
$$ \hat A=K_\infty (A),\; \hat B=K_\infty (B),\; \hat C=K_\infty (C).$$ Then,
$$ \hat A^{**}=M_\infty (A^{**}),\; \hat B^{**}=M_\infty (B^{**}),\; \hat C^{**}=M_\infty ^w(C^{**}). $$ 
There exist completely isometric homomorphisms
$$\theta : \hat A\rightarrow \hat B^{**}, \;\;\;\rho : \hat B\rightarrow \hat C^{**},$$ 
such that $$\theta(\hat A)= p^\infty \hat B p^\infty , \;\;\;\rho(\hat B)=  q^\infty \hat Cq^\infty  .$$
 There exists a completely isometric homomorphism $\rho _0: \hat B^{**}\rightarrow \hat C^{**}$ such that 
$$\rho _0|_{\hat B}=\rho , \;\;\rho _0(\hat B^{**})\subseteq q^\infty \hat C^{**}q^\infty  .$$
Because $p^\infty \in Z(\Delta (\hat B^{**}))$ and $\rho _0(p^\infty )\leq q^\infty $, 
there exists $q_0\in Z(\Delta (\hat C^{**}))$ such that $\rho _0(p^\infty )= q_0^\infty. $ Now,
$$\rho (\theta (\hat A))=\rho _0(\theta (\hat A))=\rho _0(p^\infty \hat Bp^\infty )=$$
$$\rho _0(p^\infty)\rho _0( \hat B)\rho _0(p^\infty )= q_0^\infty q^\infty \hat Cq^\infty q_0^\infty  =
q_0^\infty  \hat C q_0^\infty . $$
Thus, $$\rho \circ \theta (K_\infty (A))=K_\infty (q_0Cq_0).$$ Because $\rho \circ \theta $  is a completely isometric 
homomorphism, we have that $$A\sim _{\sigma \Delta }q_0Cq_0\Rightarrow A\subset _{\sigma \Delta }C.$$ 
\end{proof}

\begin{remark} Following this theorem, one should expect that $\subset _{\sigma \Delta }$  is a partial 
order relation in the class of operator algebras if we identify those 
operator algebras that are $\sigma $-strongly  $\Delta  $-equivalent. This means that the additional property holds that
$$A\subset_{\sigma \Delta } B, \;\;B\subset_{\sigma \Delta } A\Rightarrow A\sim _{\sigma \Delta }B.$$ 
However, this is not true, as we will prove in Section \ref{non}.   

\end{remark}

\subsection{The case of $C^*$-algebras}

In this subsection, we investigate the relation $\subset _{\sigma \Delta }$ in the case of $C^*$-algebras.

\begin{theorem}\label{21100} Let $A, B$ be $C^*$-algebras. The following are equivalent:

(i) $$A\subset _{\sigma \Delta }B$$.

(ii) There exists an onto $*$-homomorphism $\theta : K_\infty (B)\rightarrow K_\infty (A).$

(iii) There exists an ideal $I$ of $B$ such that $$A\sim _{\sigma \Delta }B/I.$$

(iv) For every $*$-isomorphism $\alpha : A\rightarrow \alpha (A)$, there exists a $*$-homomorphism (not necessarily 
faithful) $\beta: B \rightarrow \beta (B)$ such that $\alpha (A)\sim _{\sigma TRO}\beta (B).$
\end{theorem}
\begin{proof}

(i) $\Rightarrow (ii)$

By Definition \ref{2100} and Theorem \ref{225}, there exist a projection $p\in Z(B^{**})$ and a $*$-isomorphism $\rho : K_\infty (pB)\rightarrow 
K_\infty (A).$ Define the onto $*$-homomorphism $\tau : K_\infty (B)\rightarrow K_\infty (pB)$ 
given by $\tau (  (b_{i,j})_{i,j}  )=(pb_{i,j})_{i,j}.$ We denote $\theta =\rho \circ \tau .$ 

(ii) $\Rightarrow (i)$

Suppose that $\theta ^{**}: M_\infty (B^{**})  \rightarrow M_\infty (A^{**}) $ is the second dual of $\theta,$ then there 
exists a projection $q\in Z(B^{**})$ such that $$\theta^{**}(xq^\infty )= \theta^{**} (x), \;\;\forall 
\;x\in M_\infty(B^{**}),$$ and $\theta |_{M_\infty (B^{**}q)}$ is a $*$-homomorphism.
Thus, if $x\in K_\infty (B)$, we have that $\theta (xq^\infty )=\theta (x).$ Therefore, $$K_\infty (Bq)\cong K_\infty (A),$$ 
which implies that 
 $$A\sim _{\sigma \Delta }Bq\Rightarrow A\subset _{\sigma \Delta }B.$$

(iii) $\Rightarrow (ii)$

If $A\sim _{\sigma \Delta }B/I,$ then 
$$K_\infty (A)\cong K_\infty (B/I)\cong K_\infty (B)/K_\infty (I).$$ 
Because $K_\infty (I) $ is an ideal of $K_\infty (B),$ there exists an onto $*$-homomorphism $\theta : K_\infty (B)\rightarrow 
K_\infty (A).$

(ii) $\Rightarrow (iii)$

Suppose that  $\theta : K_\infty (B)\rightarrow K_\infty (A)$ is an onto $*$-homomorphism. Then, there exists an ideal $J\subseteq K_\infty 
(B)$ such that $$K_\infty (B)/J\cong K_\infty (A).$$ 
The ideal $J$ is of the form $K_\infty (I)$ for an ideal $I $ of $B.$ Thus, 
$$K_\infty (B/I)\cong K_ \infty (B)/K_\infty (I)\cong K_\infty (B)/J\cong K_\infty (A).$$
Therefore, $A\sim _{\sigma \Delta }B/I.$

(iv) $\Rightarrow (iii)$

Suppose that $\alpha : A\rightarrow \alpha (A), \beta : B \rightarrow \beta (B)$ are $*$-homomorphisms such that 
$Ker \alpha =\{0\}$ and   $\alpha (A)\sim _{\sigma TRO}\beta (B).$ Let $I$ be the  ideal $Ker \beta .$ 
Then, $\beta (B)\cong B/I,$ and thus $A\sim _{\sigma \Delta }B/I.$

(iii) $\Rightarrow (iv)$

We assume that $\alpha: A\rightarrow  \alpha (A)$ is a faithful $*$-homomorphism, and that $A\sim _{\sigma \Delta }B/I.$
 By Theorem \ref{223}, there exists a faithful $*$-homomorphism $\gamma : B/I\rightarrow \gamma (B/I)$ such that 
$$\alpha (A)\sim _{\sigma TRO}\gamma (B/I).$$ 
If $\pi : B\rightarrow B/I$ is the natural mapping and $\beta =\gamma \circ \pi $, 
then $$\alpha (A)\sim _{\sigma TRO}\beta (B).$$
\end{proof}

\begin{remark}\label{000000}\em {If   $A$ and $B$ are $W^*$-algebras and $\alpha : A\rightarrow B, \;\;\;\beta: B \rightarrow  A$ are 
$w^*$-continuous onto $*$-homomorphisms, then $A$ and $B$ are $*$-isomorphic. Indeed, there exist projections $e_1\in Z(A), f_1\in Z(B)$ 
such that $$Ae_1\cong B, \;\;\;Bf_1\cong A.$$ Thus, there exists a projection $e_2\in Z(A), e_2\leq e_1$ such that $$Ae_2\cong Bf_1\Rightarrow 
Ae_2\cong A.$$ From the proof of Lemma 2.17 in \cite{elest}, we have that $A\cong Ae_1,$ and thus $A\cong B.$ 
In Example \ref{2000000}, we will present non-isomorphic $C^*$-algebras  $A$ and $B$ for which there 
exist onto $*$-homomorphisms $\alpha : A\rightarrow B, \;\;\;\beta: B \rightarrow  A.$ These algebras 
are not $W^*$-algebras. }\end{remark}

\begin{remark}\em{ As we have previously mentioned, in \cite{elest} we defined an analogous 
 relation $\subset _\Delta $ between unital dual operator algebras. 
We have proven that if $A\subset _\Delta B,$ where $A, B$ are unital dual operator algebras, then there exists a central projection $p$ 
in $\Delta (B)$ 
and a Hilbert space $H$ such that $A\bar \otimes B(H)$ and $(pBp)\bar \otimes B(H)$ are isomorphic as dual 
operator algebras. Here, $\bar \otimes $ is the normal spatial tensor product.  
In the case of $W^*$-algebras, we have proven that $A\subset _\Delta B$ if and only if there exists a a Hilbert space $H$ and a 
$w^*$-continuous $*$-homomorphism from 
 $B\bar \otimes B(H)$  onto $ A\bar \otimes B(H)).$ We have also proven that if  $A$ and $B$ are $W^*$-algebras such that 
$A\subset _\Delta B$ and $B\subset _\Delta A$, then $A$ and $B$ are stably isomorphic in the weak sense. We present a new proof of this fact here.

Suppose that  $A\subset _\Delta B$ and $B\subset _\Delta A.$ Then, there exist Hilbert spaces $H$ and $K$ and $w^*$-continuous $*$-homomorphisms 
from $B\bar \otimes B(H) $ onto $A\bar \otimes B(H)$ and  from $A\bar \otimes B(K) $ onto $B\bar \otimes B(K).$ 
We conclude that there exist  $w^*$-continuous $*$-homomorphisms 
from $B\bar \otimes B(H)\bar \otimes B(K) $ onto $A\bar \otimes B(H)\bar \otimes B(K) $ and  
from $A\bar \otimes B(K)\bar \otimes B(H)  $ onto $A\bar \otimes B(K)\bar \otimes B(H) .$ Therefore, 
by Remark \ref{000000}, 
 $$A\bar \otimes B(K)\bar \otimes B(H) \cong B\bar \otimes B(H)\bar \otimes B(K). $$ Because 
$$B(H)\bar \otimes B(K)) \cong B(K)\bar \otimes B(H) \cong B(K\otimes  H),$$ 
we have that $$A\bar \otimes B(K\otimes  H)\cong B\bar \otimes B(K\otimes H).$$ 
Thus, $A$ and $B$ are stably isomorphic.
}
\end{remark}

\begin{remark}\em{ The relation $\subset_ \Delta $ between $W^*$-algebras is a partial order relation up to 
weak stable isomorphism \cite{elest}. This means that it has the following properties:

(i) $A\subset _\Delta A$.

(ii) $A\subset _\Delta B,\;\;\;B\subset _\Delta C\Rightarrow A\subset _\Delta C.$

(iii)If $A\subset _\Delta B$ and $B\subset _\Delta A$, then $A$ and $B$ are weakly stably isomorphic. 
Therefore, it is natural to ask whether $\subset _{\sigma \Delta }$ is a partial order relation up to 
strong  stable isomorphism for $C^*$-algebras. Although $\subset _{\sigma \Delta }$ satisfies the properties (i) and (ii), it does not 
satisfy property (iii), as we show in Example \ref{2000000}. Nevertheless, $\subset _{\sigma \Delta }$ 
satisfies the property described in Theorem \ref{vary}.}
\end{remark}

\begin{example}\label{1000000}\em{Let $X, Y$ be compact metric spaces, $\theta : X\rightarrow Y$ be a continuous one-to-one function, and 
$C(X)$ and $C(Y)$ be the algebras of continuous functions from  $X$ and $Y$, respectively, into the complex plane $\bb C$, equipped 
with the supremum norm. Then, the map $$\rho : C(Y)\rightarrow C(X), \;\;\rho (f)=f\circ \theta $$ is an onto $*$-homomorphism, 
and thus $C(X)\subset _{\sigma \Delta }C(Y).$ Indeed, if $g\in C(X)$ we define 
$$f_0:\theta (X)\rightarrow \bb C,\;\;f_0(\theta (x))=g(x).$$ Because $\theta : X\rightarrow \theta (X)$ is a homeomorphism, $f_0$ 
is continuous. By Tietze's theorem, there exists $f\in C(Y)$ such that $f|_{\theta (X)}=f_0.$ We 
have that $f\circ \theta (x)=g(x)$ for all $x\in X,$ and thus $\rho (f)=g.$   
}

\end{example}

\begin{example}\label{2000000}\em{There exist commutative $C^*$-algebras $A$ and $B$ such that $A\subset_{\sigma \Delta }B, \;\;\;
B \subset _{\sigma \Delta }A,$ but $A$ and $B$ are not strongly Morita equivalent. Thus, $A$ and $B$ are not $\sigma \Delta $-
equivalent. We denote the following subsets of $\bb C:$

$$  X=\{z\in \bb C: 1\leq |z|\leq 5\} , \;\;\;Y=\{z\in \bb C: |z|\leq 5\} .$$ 
We write $A=C(X),\;\;B=C(Y).$ Because $X\subseteq Y$, by Example \ref{1000000} we have that $A\subset_{\sigma \Delta }B.$
All of the closed discs of $\bb C$ are homeomorphic, and thus there exists a homeomorphism $\theta : Y\rightarrow X_0,$ 
where  $X_0=\{z\in \bb C: |z-3|\leq 1\}.$ Because $X_0\subseteq X,$ Example \ref{1000000} implies that $B\subset_{\sigma \Delta }A.$
 If $A$ and $B $ were strongly Morita equivalent, then they would also be $*$-isomorphic. The Stone-Banach theorem implies that $X$ and 
$Y$ would then be homeomorphic. However, this contradicts the fact that $Y$ is a simply connected set and $X$ is not. }

\end{example}

Next, we will prove  Theorem \ref{vary}, which states the following:
$$A\subset _{\sigma \Delta }B, \;\;\;B\subset _{\sigma \Delta }A\Rightarrow Ar \sim_{\sigma \Delta }B\hat r
, \;\;\;\; A(id_{A^{**}}-r) \sim _{\sigma \Delta }
 B(id_{B^{**}}-\hat r), $$ for central projections $r\in A^{**}, \hat r\in B^{**}.$ 

\begin{lemma}\label{10000} Let $A, B$ be operator algebras and $\hat A, \hat B$ be unital dual operator algebras 
such that $ \hat A=\overline{A}^{w^*} , \hat B=\overline{B}^{w^*}. $ Furthermore, let  $M$ be a TRO such that 
$$ A=\overline{[M^*BM]}^{\|\cdot\|},\;\;\; B=\overline{[MAM^*]}^{\|\cdot\|}, $$ and let $\alpha : \hat A
\rightarrow \alpha (\hat A)$ be  a $w^*$-continuous completely isometric homomorphism such that $H=\overline{
\alpha (A)(H)}.$ Then, there exist a Hilbert space $K,$ a $w^*$-continuous completely isometric honomorphism $\beta : 
\hat B\rightarrow B(K)$ such that $K=\overline{\beta (B)(K)}$, and a TRO homomorphism $\mu : M\rightarrow B(H,K)$ 
such that the following hold:
 
A) If $a\in A, b\in B, m,n\in M$ such that $a=m^*bn$, then $\alpha (a)=\mu (m)^* \beta (b)\mu (n).$

B) If $a\in A, b\in B, m,n\in M$ such that $b=man^*$, then $\beta (b)=\mu (m)\alpha (a)\mu (n)^*.$
Therefore,

$$ \alpha (\hat A)=\overline{[\mu (M)^*\beta (\hat B)\mu (M)]}^{w^*},\;\;\; \beta (\hat B)
=\overline{[\mu (M)\alpha (\hat A)\mu (M)^*]}^{w^*} $$ and 
 $$ \alpha (A)=\overline{[\mu (M)^*\beta (B)\mu (M)]}^{\|\cdot\|},\;\;\; \beta (B)
=\overline{[\mu (M)\alpha (A))\mu (M)^*]}^{\|\cdot\|} .$$ 
\end{lemma}

The proof of this lemma can be inferred from the proof of Theorem 2.12 in \cite{elehoust}, with the addition of some simple 
modifications.

\begin{definition} \label{20000} Let $\hat A, \hat B$ be von Neumann algebras, and $A$ (resp. $B$) be a $C^*$-subalgebra of 
$\hat A$ (resp. $\hat B$) such that $ \hat A=\overline{A}^{w^*}$ (resp. $\hat B=\overline{B}^{w^*},  $ ).
We write $(A, \hat A)\sim _\Delta (B, \hat B)$ if there exist $w^*$-continuous and injective $*$-homomorpisms 
$\alpha: \hat A \rightarrow \alpha (\hat A), \;\;\beta: \hat B\rightarrow 
 \beta (\hat B)$ and a $\sigma $-TRO $M$ such that \begin{equation} \label{refer} \alpha (A)=\overline{[M^*\beta (B)M]}^{\|\cdot\|} , 
\;\; \beta (B)=\overline{[M\alpha (A)M^*] }^{\|\cdot\|}. \end{equation} 
\end{definition}

\begin{remarks} \label{easy}\em{(i) If (\ref{refer}) holds then $$\alpha (\hat A)=\overline{[M^*\beta (\hat B)M]}^{w^*} , 
\;\; \beta (\hat B)=\overline{[M\alpha (\hat A)M^*] }^{w^*}. $$

(ii) Lemma \ref{10000} implies that if $(A,\hat A)\sim _\Delta (B,\hat B)$ 
and $\gamma : \hat A\rightarrow \gamma (\hat A)$ is a $w^*$-continuous $*$-isomorphism, then there exists a 
$w^*$-continuous $*$-isomorphism $\delta : \hat B\rightarrow \delta (\hat B)$ and a $\sigma $-TRO $N$ such that 
$$\gamma (A)=\overline{[N^*\delta (B)N]}^{\|\cdot\|}, \;\;\;\delta (B)=\overline{[N\gamma (A)N^*]}^{\|\cdot\|}.$$

(iii) The above remark and Theorem \ref{224} both imply that if $\hat A, \hat B, \hat C$ are 
 von Neumann algebras $A,B,C$ are, respectively, $w^*$-dense  $C^*$-subalgebras of these, and  
$(A,\hat A)\sim _\Delta (B,\hat B)$ and 
$(B,\hat B)\sim _\Delta (C,\hat C), $ 
then $(A,\hat A)\sim _\Delta (C,\hat C).$  }
\end{remarks}

In the following, we assume that $A$ is a $C^*$-algebra such that $A\subseteq A^{**}\subseteq B(H)$ for 
some Hilbert space $H$, and $e_2$ is a central projection of $A^{**}.$ We also assume that $A\sim _{\sigma \Delta 
}Ae_2.$

\begin{lemma}\label{10000a} There exist a $w^*$-continuous $*$-isomorphism $\theta _1: A^{**}\rightarrow \theta _1(A^{**})$
 and a $\sigma $-TRO $M$ such that 
 $$\theta _1(A)=\overline{[M^*Ae_2M]}^{\|\cdot\|} , \;\; Ae_2=\overline{[M\theta _1(A)M^*] }^{\|\cdot\|}. $$ 
\end{lemma}
\begin{proof} Let $B$ be a $C^*$ algebra. We assume that $B\subseteq B^{**}\subseteq B(H).$ Let $\cl K$ be the algebra of compact operators 
acting on $l^2(\bb N)$, and $p\in \cl K$ be a rank one projection. We define the $\sigma $-TRO $M=I_H\otimes p\cl K.$ Then, 
we have that $$   B\otimes p=\overline{[M(B\otimes \cl K)M^*]}^{\|\cdot\|}, \;\;\; B\otimes \cl K=\overline{[M^*(B\otimes p)M]}^{\|\cdot\|}, $$
where $\otimes $ is the minimal tensor product. Because $ \overline{B\otimes p}^{w^*}=B^{**}\bar \otimes p , \;\;\;
\overline{B\otimes \cl K}^{w^*}=B^{**}\bar \otimes B(l^2(\bb N))$, here $\bar \otimes $ is the spatial tensor product, we 
have $$(B\otimes p, B^{**}\bar \otimes p)\sim _\Delta (B\otimes \cl K, B^{**}\bar \otimes B(l^2(\bb N))).$$

Because there exists a $*$-isomorphism from $B^{**}$ onto $B^{**}\bar \otimes p$ mapping 
$B$ onto $B \otimes p$, we can conclude that  $(B, B^{**})\sim _\Delta (B\otimes \cl K, B^{**}\bar \otimes B(l^2(\bb N))).$

Therefore, $$(A, A^{**})\sim _\Delta ( A\otimes \cl K , A^{**}\bar \otimes B(l^2(\bb N))   )$$  and  $$(Ae_2, A^{**}e_2) \sim _\Delta 
 ((Ae_2)\otimes \cl K, (A^{**}e_2)\bar \otimes B(l^2(\bb N))).$$

Because $A\sim _{\sigma \Delta }Ae_2$, there exists a $*$-isomorphism from $A^{**}\bar \otimes B(l^2(\bb N))$ onto 
$(A^{**}e_2)\bar \otimes B(l^2(\bb N))$ mapping  $A\otimes \cl K$ onto  $(Ae_2)\otimes \cl K $ 
and, therefore, $$( A\otimes \cl K , A^{**}\bar \otimes B(l^2(\bb N))   )\sim _\Delta  ((Ae_2)\otimes \cl K, (A^{**}e_2)\bar \otimes B(l^2(\bb N))).$$
 Now Remark \ref{easy}, (iii), implies that $(A, A^{**})\sim _\Delta (Ae_2, A^{**}e_2).$ By Remark \ref{easy}, (ii), 
for the identity map $id: A^{**}e_2 \rightarrow A^{**}e_2 $ there exist a $w^*$-continuous $*$-isomorphism $\theta _1: A^{**}\rightarrow \theta _1(A^{**})$
 and a $\sigma $-TRO $M$ such that 
 $$\theta _1(A)=\overline{[M^*Ae_2M]}^{\|\cdot\|} , \;\; Ae_2=\overline{[M\theta _1(A)M^*] }^{\|\cdot\|}. $$ 
\end{proof}

\begin{lemma}\label{30000}Let $M, \theta _1$ be as in Lemma \ref{10000a}. Then, there exist $w^*$-continuous $*$-
isomorphisms $\rho _k: A^{**}\rightarrow \rho _k(A^{**})$ and TRO homomorphims $\phi _k: M\rightarrow \phi _k(M), 
k=0,1,2,...$ where $\rho _0=id_{A^{**}}, \phi _0=id_M,$ such that if $a\in A^{**}, x\in A^{**}e_2, m,n \in M,$ the equality 
 $\rho _k(a)=\phi _{k-1}(m)^*\rho _{k-1}(x) \phi _{k-1}(n)$  implies that 
$\rho _{k+1}(a)=\phi _{k}(m)^*\rho _{k}(x) \phi _{k}(n)$ and the equality $\rho _{k-1}(x)=\phi _{k-1}(m)\rho _{k}(a) \phi _{k-1}(n)^*$  implies that
 $\rho _{k}(x)=\phi _{k}(m)\rho _{k+1}(a) \phi _{k}(n)^*$ for all $k=1,2,...$   Therefore, 
$$\rho _k( A^{**})=\overline{[  \phi _{k-1}(M)^*\rho _{k-1}( A^{**}e_2)\phi _{k-1}(M)   ]}^{w^*} ,$$$$ 
 \rho _{k-1}(A^{**}e_2)=\overline{[   \phi _{k-1}(M)\rho _k(A^{**})\phi _{k-1}(M)^*    ] }^{w^*} $$ 
and $$\rho _k(A)=\overline{[\phi _{k-1}(M)^*\rho _{k-1}( Ae_2)\phi _{k-1}(M)   ]}^{\|\cdot\|} ,$$$$ 
 \rho _{k-1}(Ae_2)=\overline{[  \phi _{k-1}(M)\rho _k(A)\phi _{k-1}(M)^*   ] }^{\|\cdot\|} $$
for all $k=1,2,...$ 
\end{lemma}
\begin{proof} By Lemma \ref{10000}, given the representation $\theta _1|_{A^{**}e_2}$, there exists a $*$-isomorphism 
$$\theta _2: \theta _1(A^{**})\rightarrow \theta_2( \theta_1(A^{**})) $$
and a TRO homomorphism $\phi _1: M\rightarrow \phi _1(M)$ such that
$$\theta _2(\theta _1( A^{**}))=\overline{[\phi _1(M)^*\theta _1( A^{**}e_2)\phi _1(M)]}^{w^*} , 
\;\; \theta _1(A^{**}e_2)=\overline{[\phi _1(M)\theta _2(\theta _1(A^{**}))\phi _1(M)^*] }^{w^*} $$ 
and $$\theta _2(\theta _1(A))=\overline{[\phi _1(M)^*\theta _1(Ae_2)\phi _1(M)]}^{\|\cdot\|} , 
\;\; \theta _1(Ae_2)=\overline{[\phi _1(M)\theta _2(\theta _1(A))\phi _1(M)^*] }^{\|\cdot\|}, $$
and such that if  $a\in A^{**}, x\in A^{**}e_2, m,n \in M$, the equality $\theta _1(a)=m^*xn$ implies that $\theta_2( \theta_1(a))=
\phi _1(m)^*\theta _1(x)\phi _1(n) $ and the equality $x=m\theta _1(a)n^*$ implies that $\theta _1(x)=\phi _1(m)\theta_2( \theta_1(a))
\phi _1(n)^*. $ 

We write $\rho _0=id_{A^{**}}, \rho _1=\theta _1, \rho _2=\theta_2 \circ \theta _1$ and continue inductively.
 \end{proof}
Let $M, \theta _1$ be as in Lemma \ref{10000a}. Given the $*$-isomorphism $\theta _1^{-1}: \theta _1(A^{**})\rightarrow A^{**},$ Lemma
\ref{10000} implies that there exist a $*$-isomorphism $\sigma _1: A^{**}e_2\rightarrow \sigma _1(A^{**}e_2)$ and a TRO 
homomorphism $\chi_0 : M\rightarrow \chi _0(M)$ such that if $\chi (m)=\chi_0(m)^*,\;\;\forall \;m\;\in \;M, $ then  
$$ A^{**}=\overline{[\chi (M)\sigma _1( A^{**}e_2)\chi (M)^*]}^{w^*} , \;\; \sigma _1(A^{**}e_2)=
\overline{[\chi (M)^*A^{**}\chi (M)] }^{w^*} $$ 
and $$A=\overline{[\chi (M)\sigma _1(Ae_2)\chi (M)^*]}^{\|\cdot\|} , \;\; \sigma _1(Ae_2)=
\overline{[\chi (M)^*A\chi (M)] }^{\|\cdot\|} .$$ Furthermore, if $a\in A^{**}, m,n\in M, x\in A^{**}e_2$ then 
the equality $\theta _1(a)=m^*xn$ implies that $a=\chi(m)\sigma  _1(x) \chi (n)^*.$

\begin{lemma}\label{40001}  Let $M, \chi , \theta _1$ be as in the previous discussion, then there exists a $w^*$-continuous $*$-isomorphism  $\tau _1: A^{**}\rightarrow \tau _1(A^{**})$ and a TRO homomorphism 
$\psi _1: \chi (M)\rightarrow \psi _1(\chi (M))$ such that if $a\in A^{**}, m,n\in M, x\in A^{**}e_2$, then the
equality $a=\chi(m)\sigma  _1(x) \chi (n)^*$ implies that $a=\psi _1(\chi (m))\tau _1(x)\psi _1(\chi (n))^*$ 
and $\sigma _1(x)=\chi (m)^*a \chi(n) $ implies that $\tau _1(x)=\psi _1(\chi (m))^*a \psi _1(\chi (n)). $
Thus, $$ A^{**}=\overline{[  \psi _{1}(\chi (M))\tau _{1}( A^{**}e_2)\psi _{1}(\chi (M))^*   ]}^{w^*} ,$$
$$ A=\overline{[  \psi _{1}(\chi (M))\tau _{1}( Ae_2)\psi _{1}(\chi (M))^*   ]}^{\|\cdot\|} ,$$
$$ \tau _{1}(A^{**}e_2)=\overline{[ \psi _{1}(\chi (M))^*A^{**}\psi _{1}(\chi (M))    ] }^{w^*} $$ 
$$ \tau _{1}(Ae_2)=\overline{[ \psi _{1}(\chi (M))^*A\psi _{1}(\chi (M))    ] }^{\|\cdot\|} .$$ 
\end{lemma}
\begin{proof} Define the $*$-isomorphism $\tau _1: A^{**}\rightarrow \sigma _1(A^{**}e_2)\oplus A^{**}e_2^\bot ,$ 
given by $\tau _1(a)=\sigma _1(ae_2)\oplus ae_2^\bot ,$ and the TRO homomorphism $\psi_1: \chi (M)\rightarrow  \psi_1(\chi (M)) $ given by  
$\psi _1(\chi (m))=(\chi (m) \;\;0).$ If  $a\in A^{**}, m,n\in M, x\in A^{**}e_2$ satisfies   
$a=\chi(m)\sigma  _1(x) \chi (n)^*$, then
$$a=(\chi (m) \;\;0) \left (\begin{array}{clr}\sigma _1(x) & 0 \\ 0 & 0 \end{array}\right)
(\chi (n)^*\;\; 0)^t=\psi _1(\chi (m))\tau _1(x)\psi _1(\chi (n))^*.$$ 
 Furthermore, if $\sigma _1(x)=\chi (m)^*a \chi(n) $, then 
$$\tau _1(x)=\left(\begin{array}{clr}\sigma _1(x) & 0 \\ 0 & 0 \end{array}\right)
=\left(\begin{array}{clr}\chi (m)^*a\chi (n) & 0 \\ 0 & 0\end{array}\right)=$$$$(\chi (m)^* \;\;0)^ta(\chi (n) \;\;0)= 
\psi _1(\chi (m))^*a \psi _1(\chi (n)). $$ 
\end{proof}

\begin{lemma}\label{40002}  Let $\tau _1, M, \chi , \psi _1$ be as in Lemma \ref{40001}. Then, there exist $w^*$- continuous $*$-isomorphisms  $\tau _k: A^{**}\rightarrow \tau _k(A^{**})$ and  TRO homomorphisms 
$\psi _k:\chi ( M)\rightarrow \psi _k(\chi (M))$ such that if $a\in A^{**}, m,n\in M, x\in A^{**}e_2$ the 
equality $a=\psi _1(\chi (m))\tau  _1(x) \psi _1(\chi (n))^*$ implies that $\tau _k(a)=\psi _{k+1}(\chi (m))
\tau _{k+1}(x)\psi _{k+1}(\chi (n))^*$ 
and $\tau _1(x)=\psi _1(\chi (m))^*a \psi _1(\chi (n)) $ implies that $\tau _{k+1}(x)=\psi _{k+1}(\chi (m))^*\tau _k(a) \psi _{k+1}(\chi (n))$
 for all $k=1,2,\ldots$. Thus, $$ \tau _k(A^{**})=\overline{[  \psi _{k+1}(\chi (M))
\tau _{k+1}( A^{**}e_2)\psi _{k+1}(\chi (M))^*   ]}^{w^*} ,$$
$$ \tau _k(A)=\overline{[  \psi _{k+1}(\chi (M))\tau _{k+1}( Ae_2)\psi _{k+1}(\chi (M))^*   ]}^{\|\cdot\|} ,$$
$$ \tau _{k+1}(A^{**}e_2)=\overline{[ \psi _{k+1}(\chi (M))^*\tau _k(A^{**})\psi _{k+1}(\chi (M))  ] }^{w^*} $$ 
$$ \tau _{k+1}(Ae_2)=\overline{[ \psi _{k+1}(\chi (M))^*\tau _k(A)\psi _{k+1}(\chi (M))    ] }^{\|\cdot\|} .$$ 
\end{lemma}
\begin{proof}Lemma \ref{10000} implies that given the $*$-isomorphism $\tau_1: A^{**}\rightarrow  \tau _1(A^{**})$, 
there exist a $w^*$-continuous $*$-isomorphism $\tau_{2,0}: A^{**}e_2\rightarrow  \tau _{2,0}(A^{**}e_2)$ 
and a TRO homomorphism $\zeta : \chi (M)\rightarrow \zeta (\chi (M)) $ such that if $a\in A^{**}, m,n\in M, x\in A^{**}e_2$, then the
equality $a=\psi _1(\chi (m))\tau _1(x) \psi _1(\chi (n))^*$ implies that $\tau _1(a)=\zeta (\chi (m))\tau _{2,0}(x)\zeta (\chi (n))^*$ 
and $\tau _1(x)=\psi _1(\chi (m))^*a \psi _1(\chi ( n)) $ implies that $\tau_{2,0}(x)=\zeta (\chi (m))^*\tau _1(a)\zeta (\chi (n)). $ 
For every $a\in A^{**}, m\in M,$ we define
$$\tau _2(a)=\tau _{2,0}(ae_2)\oplus ae_2^\bot ,\;\;\;\psi _2(\chi (m))=(\zeta (\chi (m)) \;\; 0).$$
If $a\in A^{**}, m,n\in M, x\in A^{**}e_2$, then the
equality $a=\psi _1(\chi (m))\tau _1(x) \psi _1(\chi (n))^*$ implies that
$$\tau _1(a)=\zeta (\chi (m))\tau _{2,0}(x)\zeta (\chi (n))^*=(\zeta (\chi (m)) \;\; 0) \left(\begin{array}{clr}
 \tau _{2,0}(x) &0 \\ 0& 0\end{array}\right)(\zeta (\chi (n))^*\;\; 0)^t
=$$ $$ \psi _2(\chi (m))\tau _2(x)\psi _2(\chi (n))^*$$ 
and the equality $\tau _1(x)=\psi _1(\chi (m))^*a \psi _1(\chi (n)) $ implies that
$$\tau _2(x)=\left(\begin{array}{clr} \tau _{2,0}(x) &0 \\ 0& 0\end{array}\right)=
\left(\begin{array}{clr} \zeta (\chi (m))^*\tau _{1}(a)\zeta (\chi (n))  &0 \\ 0& 0\end{array}\right)=$$$$
(\zeta (\chi (m))^* \;\; 0)^t\tau _1(a)(\zeta (\chi (n)) \;\; 0)=\psi _2(\chi (m))^*\tau _1(a)\psi _2(\chi (n)). $$ 
We continue inductively.
\end{proof}

\begin{lemma}\label{50000} There exist a faithful $*$-homomorphism $\alpha : A^{**}\rightarrow B(L),$ 
where $L$ is a Hilbert space such that $\overline{\alpha (A)(L)}=L,$ and a $\sigma $-TRO $N\subseteq B(\alpha (e_2)(L), L)$ such that 
$$\alpha (A^{**})=\overline{[N\alpha (A^{**}e_2)N^*]}^{w^*} , \;\; \alpha (A^{**}e_2)=\overline{[N^*\alpha (A^{**})N] }^{w^*} $$ 
and $$\alpha (A)=\overline{[N\alpha (Ae_2)N^*]}^{\|\cdot\|} , \;\; \alpha (Ae_2)=\overline{[N^*\alpha (A)N] }^{\|\cdot\|}. $$
\end{lemma}
\begin{proof} We recall the maps $\theta _1, \tau _k, \rho _k$ from Lemmas \ref{30000}, \ref{40001}, and \ref{40002}. We denote
$$\alpha (a)=\ldots\oplus \tau _2(a)\oplus \tau _1(a)\oplus a\oplus \rho _1(a)\oplus \rho _2(a)\oplus \ldots$$
for all $a\in A^{**}.$ We also recall the maps $\psi _k, \phi _k, \chi $, and for each $m\in M,$ we let $\zeta (m)$ 
be the $\infty \times \infty $  matrix whose first diagonal under the main diagonal is 
$$(\ldots, \psi _2(\chi (m)),\;\; \psi _1(\chi (m)),\;\;  m^*,\;\; \phi _1(m)^*,\;\;  \phi _2(m)^*, \ldots)$$
where the other diagonals have zero entries. Clearly, $\zeta (M)$ is a  $\sigma $-TRO.

Let $a\in A^{**}, x\in A^{**}e_2, m,n \in M$ be such that $\rho _1(a)=\theta _1(a)=m^*xn. $ Then, by Lemma \ref{30000} 
we have that $$\rho _{k+1}(a)=\phi _k(m)^*\rho _k(x)\phi _k(n),\;\;\forall \;k=1,2,3,...$$
Furthermore, following the discussion for the previous Lemma \ref{40001}, we have that
$a=\chi (m)\sigma _1(x)\chi (n)^*,$ which by Lemma \ref{40001} implies that $a=\psi _1(\chi (m))\tau _1(x)\psi _1(\chi (n))^*
.$ By Lemma \ref{40002}, we have that
$$\tau _k(a)=\psi _{k+1}(\chi (m))\tau _{k+1}(x)\psi _{k+1}(\chi (n))^*, \;\;\forall \;k=1,2,3,...$$
Therefore,
\begin{align*}& \zeta (m)\alpha (x)\zeta (n)^*=\\
& \ldots \psi _2(\chi (m))\tau _2(x)\psi _2(\chi (n))^* \oplus  \psi _1(\chi (m))\tau _1(x)\psi _1(\chi (n))^*\oplus 
\\ &  m^*xn\oplus   \phi _1(m)^*\rho _1(x)\phi _1(n)  \oplus 
\phi _2(m)^*\rho _2(x)\phi _2(n)  \oplus \ldots=\\& 
\ldots \oplus \tau _1(a)\oplus a\oplus \rho _1(a)\oplus \rho _2(a)\oplus \rho _3(a)\ldots=\alpha (a).
 \end{align*}
 We conclude that $$ \alpha (A^{**})=\overline{[\zeta (M)\alpha (A^{**}e_2)\zeta (M)^*]}^{w^*} ,\;\;\; 
\alpha (A)=\overline{[\zeta (M)\alpha (Ae_2)\zeta (M)^*]}^{\|\cdot\|}.$$ 
Similarly, we can see that $$ \alpha (A^{**}e_2)=\overline{[\zeta (M)^*\alpha (A^{**})\zeta (M)] }^{w^*},\;\; 
 \alpha (Ae_2)=\overline{[\zeta (M)^*\alpha (A)\zeta (M)] }^{\|\cdot\|}. $$
\end{proof}

\begin{lemma}\label{60000} Let $A$ be a $C^*$ algebra and $e_1, e_2\in Z(A^{**})$ be projections such that $Ae_2$ is a $C^*$-
algebra, $A\sim _{\sigma \Delta }Ae_2$, and $e_2\leq e_1\leq e_0=id_{A^{**}}, e_2\neq e_1\neq e_0.$ 
Then, there exist central projections $q, p, r\in A^{**}$ such that 
$$e_0=p\oplus q,\; e_1=r\oplus q,\; p\bot q,\; r\bot q,$$ and $$Ap\sim _{\sigma \Delta }Ar.$$  
\end{lemma}
\begin{proof} From Lemma \ref{50000}, we may assume that $$A\subseteq A^{**}\subseteq B(H), e_0=I_H$$ and there exists a $\sigma $-TRO $M\subseteq B(e_2(H), H)$ such that $$A^{**}=\overline{[M A^{**}e_2M^*]}^{w^*} , \;\; A^{**}e_2=  \overline{[M^*A^{**}M] }^{w^*}  $$ 
and $$A=\overline{[M Ae_2M^*]}^{\|\cdot\|} , \;\; Ae_2=\overline{[M^*AM] }^{\|\cdot\|}. $$
By Proposition 2.8 and Theorem 3.3 in \cite{eletro}, there exists a $*$-isomorphism $$\phi : (A^{**})^\prime \rightarrow (A^{**})^\prime e_2\subseteq B(e_2(H))$$
such that $$am=m\phi (a),\;\;\forall \;a\in (A^{**})^\prime ,\;\;m\in M.$$
By induction, there exist central projections $\{e_n: n\in \bb N\}\subseteq A^{**}$ such that 
$$\phi (e_n)=e_{n+2}, \;\;\;e_{n+1}\leq e_n, \;\;\forall n=0,1,2,\ldots$$
Define $$p= \sum_{n=0}^\infty (e_{2n}-e_{2n+1}) , \;\;\; q=\sum_{n=0}^\infty (e_{2n+1}-e_{2n+2})\oplus (\wedge _n e_n). $$
Then, $e_0=p\oplus q.$ If $$r=\phi (p)=\sum_{n=0}^\infty (e_{2n+2}-e_{2n+3}), $$
then $e_1=r\oplus q.$  We define $N=pMr.$ 
Because $$NN^*N=pMrM^*pMr=pM\phi (p)M^*pM\phi (p)=pMM^*M \phi (p)\subseteq pMr=N,$$ 
 $N$ is a TRO. Furthermore, the fact that $M$ is a $\sigma $-TRO implies that $N$ is a $\sigma $-TRO. We have that $$Ar=A\phi (p)=
Ae_2\phi (p)= \overline{[\phi (p)M^*AM\phi (p)] } ^{\|\cdot\|} .$$
Thus, because $pM=M\phi (p),$ we have that $$Ar=
\overline{[N^*ApN]} ^{\|\cdot\|}. $$ Similarly, we can prove that $Ap=\overline{[NArN^*]}^{\|\cdot\|}.$
 Therefore, $Ap \sim _{\sigma \Delta }Ar.$ 
\end{proof}

\begin{theorem}\label{vary} Let $A, B$ be $C^*$-algebras such that $A\subset _{\sigma \Delta }B, \;\;\;B\subset _{\sigma \Delta }A.$ 
Assume that $e_0=id_{A^{**}}, \;\hat e_0=id_{B^{**}}.$  Then, there exist projections 
$ r\in Z(A^{**}) , \hat r\in Z(B^{**}) $
such that 
$$ Ar\sim _{\sigma \Delta} B\hat r  ,\;\;\;A(e_0-r)\sim _{\sigma \Delta} B(\hat e_0-\hat r ). $$
\end{theorem}
\begin{proof}
There exist projections $e_1\in Z(A^{**}), f_1\in Z(B^{**})$ such that  
$$A\sim _{\sigma \Delta }Bf_1,\;\;B\sim _{\sigma \Delta }Ae_1.$$ 
Thus, there exists a projection $e_2\in  Z(A^{**})$ such that $e_2\leq e_1$ and $Bf_1\sim _{\sigma \Delta }Ae_2.$ 
Therefore, $A\sim _{\sigma \Delta }Ae_2.$ By Lemma \ref{60000}, there exist projections $p,q,r\in Z(A^{**})$ such that 
$$e_1=p\oplus q,\; e_0=r\oplus q,\; p\bot q,\; r\bot q$$ and $$Ap\sim _{\sigma \Delta }Ar.$$ 
Assume that $\psi : K_\infty (Ae_1)\rightarrow K_\infty (B)$ is a $*$-isomorphism. Again, by $\psi $ we denote
the second dual of $\psi .$ Because $p\leq e_1, $ there exists $\hat p \in Z(B^{**})$ such that 
$\psi (K_\infty (Ae_1)^{**}p^\infty )=K_\infty (B)^{**}\hat p^\infty .$ 
We have that $$\psi (K_\infty (Ap))=\psi (K_\infty (Ae_1)) \psi (p^\infty )=K_\infty (B\hat p).$$
Similarly, there exists a projection $\hat q\in Z(B^{**})$ such that  $$\psi (K_\infty (Aq))=
K_\infty (B\hat q).$$ Because $p\perp q$, we have that $\hat p \bot \hat q.$ 
Furthermore, because $e_1=p\oplus q\Rightarrow \hat e_0=\hat p\oplus \hat q,$ 
we conclude that $$Ar\sim _{\sigma \Delta }Ap\sim _{\sigma \Delta }B\hat p$$
and $$A(e_0-r)=Aq\sim _{\sigma \Delta }B\hat q=B(\hat e_0-\hat p).$$ 
We write $\hat r $ for $\hat p.$ 
The proof is now complete.
\end{proof} 

\section{Examples in the non-self-adjoint case}\label{non}

In this section, we will present a counterexample of two non-self-adjoint operator algebras $\hat A, \hat B$ 
such that $\hat A\subset _{\sigma \Delta }\hat B, \;\; \hat B\subset _{\sigma \Delta }\hat A$ but $\hat A$ and  $ \hat B$ 
are not $\sigma \Delta $-strongly equivalent. 

Let $\cl N, \cl M$ be nests acting on the separable Hilbert spaces $H_1$ and $K_1$, respectively. These nests are 
called similar if  there exists an invertible operator $s: H_1\rightarrow K_1$ such that $$\cl M=\{sn(H_1): n\in \cl N\}.$$ 
In this case, the map $$\theta _s: \cl N\rightarrow \cl M, \;\;\theta _s(n)=sn(H_1)$$ 
is a nest isomorphism. This means that $\theta _s$ is one-to-one, onto, and order-preserving. We can easily check that 
$\Alg {\cl M}=s\Alg {\cl N}s^{-1}.$ If $n\in \cl N,$ we write 
$$n_-=\vee \{l\in \cl N: l\leq n, l\neq n\}.$$
In the case where $n_-$ is strictly contained in $n,$ the projection $a=n-n_-$ is called an atom of $\cl N.$ 

\begin{theorem}\label{22100} \cite[13.20]{dav} The nests $\cl N, \cl M$ are similar if and only if there exists a nest isomorphism 
$\theta : \cl N\rightarrow \cl M$ such that $$  dim((n-n_-)(H_1)))  = dim((\theta (n)-\theta (n_-))(H_2)))  $$ 
for all $n\in \cl N.$
\end{theorem} 

The lemmas \ref{22200} and \ref{22300} can be inferred from Section 13 in \cite{dav}. We present the proofs here for completeness.

\begin{lemma}\label{22200} Let $\cl N, \cl M$ be separably-acting nests, and $\theta : \cl N\rightarrow \cl M$ 
be a nest isomorphism preserving the dimensions of the atoms. For every $0<\epsilon <1$, there exists an invertible 
operator $s,$ a unitary $u$, and a compact operator $k$ such that 
$$s=u+k, \;\;\|k\|<\epsilon , \;\;\|s^{-1}\|<1+\epsilon $$ and $\theta =\theta_s. $
\end{lemma}
\begin{proof} By Theorem \cite[13.20]{dav}, there exists a compact operator $k,$ a unitary $u$, an invertible 
operator $s=u+k,$ such that  $\theta =\theta_s $ and $\|k\|<\frac{\epsilon }{1+\epsilon }.$ Observe that $\|k\|<\epsilon .$ 
We have that $u^*s=I+u^*k\Rightarrow \|I-u^*s\|<\epsilon .$ 
Therefore, 
\begin{equation}\label{star} (u^*s)^{-1}= \sum_{n=0}^\infty  (I-u^*s)^n= \sum_{n=0}^\infty  (-u^*k)^n .
\end{equation}
We conclude that $$s^{-1}=\sum_{n=0}^\infty  (-u^*k)^n u^*.$$
We have that $$\|s^{-1}\|\leq \sum_{n=0}^\infty  \|u^*k\|^n=\sum_{n=0}^\infty \|k\|^n=\frac{1}{1-\|k\|} <1+\epsilon .  $$ 

\end{proof}

\begin{lemma}\label{22300} Let $\cl N, \cl M$ be separably-acting nests and $\theta : \cl N\rightarrow \cl M$ 
be a nest isomorphism preserving the dimensions of the atoms. For every $0<\epsilon <1$, there exists an invertible 
operator $s,$ a unitary $u,$ and  compact operators $k,l$ such that 
$$s=u+k, \;\;s^{-1}=u^*+l, \;\;\|k\|<\epsilon , \;\;\|l\|<\epsilon $$ and $\theta =\theta_s. $
\end{lemma}
\begin{proof} Choose $0<\delta <1$ such that $(1+\delta )\delta <\epsilon ,\;\; \delta <\epsilon .$ 
By Lemma \ref{22200}, there exist a unitary $u$ and compact $k$ such that $s=u+k$ is an invertible operator and
$\|k\|<\delta , \;\|s^{-1}\|<1+\delta , \;\theta =\theta_s. $ 
Define $l_0=-u^*ks^{-1}u.$ We have that
$$l_0u^*s=-u^*k\Rightarrow l_0(I+u^*k)=-u^*k\Rightarrow I=I+u^*k+l_0(I+u^*k)\Rightarrow$$$$ I=(I+l_0)(I+u^*k).$$ 
Because $\|u^*k\|<\delta <1,$ the operator $I+u^*k$ is invertible, and thus $$I+l_0=(I-(-u^*k))^{-1}=
\sum_{n=0}^\infty (-u^*k)^n.$$ By (\ref{star}), we have that
$$I+l_0=s^{-1}u\Rightarrow s^{-1}=u^*+l_0u^*.$$ If $l=l_0u^*,$ then $l$ is a compact operator, 
and \begin{align*} \|l\|=&\|l_0u^*\|=\|s^{-1}-u^*\|=\|s^{-1}u-I\|=\|s^{-1}(u-s)\|=\\& \|s^{-1}k\|\leq \|s^{-1}\|\|k\|<(1+\delta )\delta <
\epsilon .
\end{align*} Thus, $s^{-1}=u^*+l$ and $\|l\|<\epsilon.$ 
\end{proof}

In the following, we fix similar nests $\cl N$ and $\cl M$ acting on the Hilbert spaces $H_1$ and $H_2$, respectively, and a nest 
isomorphism $\theta : \cl N\rightarrow \cl M$ preserving the dimensions of atoms. 
Suppose that $a_i=n_i-(n_i)_-, \;\;b_i=\theta (n_i)-\theta (n_i)_-, i=1,2,3,...$ 
are the atoms of $\cl N$ and $\cl M$, respectively. We also assume that $p=\vee _ia_i$, $p$ is strictly 
contained in $I_{H_1}, I_{H_2}=\vee _ib_i$, and $dim(a_i)=dim(b_i)<+\infty $ for all $i.$ 
By Lemma \ref{22300}, there exists a sequence of invertible operators $(s_n)_n $ such that $\theta =\theta_{s_n}, $ 
 a sequence of unitary $(u_n)_n$, and sequences of compact operators $(k_n)_n, (l_n)_n$ such that 
$$s_n=u_n+k_n,\;\;s_n^{-1}=u_n^*+l_n$$ for all $n\in \bb N $ and $\|k_n\|\rightarrow 0, \|l_n\|\rightarrow 0.$ 
We can also assume that $\|s_n\|<2, \;\;\|s^{-1}_n\|<2 $ for all $n \in \bb N,$  and 
$$ w^*-\lim_n u_n=s , \;\; w^*-\lim_n s_n=s,\;\; w^*-\lim_n s_n^{-1}=s^*.$$

\begin{lemma}\label{22400} (i) $$SOT-\lim_n  \sum_{i=1}^\infty b_is_na_i = \sum_{i=1}^\infty b_isa_i =s_0.$$
(ii) $$SOT-\lim_n  \sum_{i=1}^\infty a_is_n^{-1}b_i = \sum_{i=1}^\infty a_is^*b_i =s_0^*.$$
\end{lemma}
\begin{proof} We shall prove (i), while statement (ii) follows by symmetry. Fix $i\in \bb N$, 
and assume that $$b_i(\xi )= \sum_{j=1}^n\sca{\xi , x_j}y_j ,\;\;\forall \;\xi \;\in \;H_1$$ for $x_j, y_j\in H_1.$
For all $\xi \in H_1$, we have that
$$b_is_na_i(\xi )= \sum_{j=1}^n\sca{s_na_i(\xi) , x_j}y_j  \rightarrow \sum_{j=1}^n\sca{sa_i(\xi) , x_j}y_j =b_isa_i(\xi ). $$
Thus, $$SOT-\lim_n b_is_na_i=b_isa_i, \;\;\;\forall \;i.$$
If $\xi \in H_1$  for all $k\in \bb N,$ we have that
\begin{align*}&  \nor{  \sum_{i=1}^\infty b_is_na_i(\xi )  - \sum_{i=1}^\infty b_isa_i(\xi )}^2 =
  \sum_{i=1}^\infty \nor{b_is_na_i(\xi )  -  b_isa_i(\xi )}^2  =\\& 
\sum_{i=1}^k \nor{b_is_na_i(\xi )  -  b_isa_i(\xi )}^2 + \sum_{i>k}^\infty \nor{b_is_na_i(\xi ) 
 -  b_isa_i(\xi )}^2 \leq \\ & 
 \sum_{i=1}^k \nor{b_is_na_i(\xi )  - b_isa_i(\xi )}^2 + 2 \sum_{i>k}^\infty \nor{a_i(\xi )}^2 .
\end{align*}
Fix $\epsilon >0.$ Then, there exists $k_0\in \bb N$ such that $\sum_{i>k_0}^\infty \nor{a_i(\xi )}^2 <\epsilon .$
Thus, $$\nor{  \sum_{i=1}^\infty b_is_na_i(\xi )  - \sum_{i=1}^\infty b_isa_i(\xi )}^2\leq 
 \sum_{i=1}^{k_0} \nor{b_is_na_i(\xi )  - b_isa_i(\xi )}^2 + 2 \epsilon ,\;\;\forall \;n\;\in \;\bb N.$$
We let $n\rightarrow \infty ,$ and we have that
 $$\limsup_n \nor{  \sum_{i=1}^\infty b_is_na_i(\xi )  - \sum_{i=1}^\infty b_isa_i(\xi )}^2\leq 0+2\epsilon =2\epsilon . $$
Thus, $$\lim_n \nor{  \sum_{i=1}^\infty b_is_na_i(\xi )  - \sum_{i=1}^\infty b_isa_i(\xi )}=0.$$ 
\end{proof}

\begin{lemma}\label{22500} For every $j, i\in \bb N$, we have that
$$a_is_j^{-1}b_is_ja_i=a_i=a_is^*b_isa_i.$$
\end{lemma}
\begin{proof}  Because $$  s_j(n_i(H_1))=\theta (n_i)(H_2)  , \;\;\;s_j((n_i)_-(H_1))=\theta ((n_i)_-)(H_2) , $$
if $\xi \in a_i(H_1)$, then $\xi=n_i( \xi)-(n_i)_-( \xi). $ Thus, there exist $\xi_j, \; \omega_j\;\in H_2, $ such that 
\begin{align*} s_j(\xi )= & \theta (n_i) (\xi _j)-\theta (n_i) _-(\omega _j) =(\theta (n_i)-\theta (n_i)_-)(\xi _j)+
(\theta (n_i)_- (\xi _j)-\theta (n_i) _-(\omega _j) )=\\& b_i(\xi _j)+\theta (n_i)_-(\xi_j- \omega _j).
\end{align*}
Because $b_i=\theta(n_i)- \theta(n_i)_-, $ we have that $b_is_j(\xi )=b_i(\xi _j).$ Therefore,
$$s_j^{-1}b_is_j(\xi )=s_j^{-1}(b_i(\xi _j))=s_j^{-1}(s_j(\xi )-\theta (n_i)_-(\xi_j- \omega _j))=
\xi - s_j^{-1}(\theta (n_i)_-(\xi _j-\omega _j)) .$$
However, $s_j^{-1}(\theta (n_i)_-(H_2))=(n_i)_-(H_1).$ Thus, there exists $\phi _j\;\in \;H_1,$ such that 
$$ s_j^{-1}(\theta (n_i)_-(\xi _j-\omega _j))=  (n_i)_-(\phi _j)$$ 
We have proved that $$s_j^{-1}b_is_j(\xi )=\xi -(n_i)_-(\phi _j),$$ which implies that
$$a_is_j^{-1}b_is_j(\xi )=a_i(\xi )-a_i(n_i)_-(\phi _j).$$ 
Because $a_i(n_i)_-=0$, we have that
$$a_is_j^{-1}b_is_j(\xi )=a_i(\xi ), \;\;\forall \;i,j.$$ Because 
$$  SOT-\lim_ja_is_j^{-1}b_i=a_is^*b_i,  \;\;\; SOT-\lim_jb_is_ja_i=b_is^*a_i,$$ we obtain 
 $$a_is^*b_isa_i= a_i, \;\;\forall \;i.$$

\end{proof}

\begin{lemma}\label{22600}  Let $s_0$ be as in Lemma \ref{22400}. Then, $$s_0^*s_0=p, \;\;s_0s_0^*=I_{H_2}.$$
\end{lemma}
\begin{proof} We shall prove that $s_0^*s_0=p.$ Because the span of the atoms of $\cl M$ 
is $I_{H_2},$  the other equality follows from symmetry. By Lemma \ref{22400}, we have that
 $$s_0=SOT-\lim_n  \sum_{i=1}^n b_isa_i, \;\;\;s_0^*=SOT-\lim_n  \sum_{i=1}^n a_is^*b_i.$$
Thus, $$s_0^*s_0= SOT-\lim_n( \sum_{i=1}^n a_is^*b_i )(\sum_{j=1}^nb_jsa_j )=
SOT-\lim_n \sum_{i=1}^n a_is^*b_isa_i.$$
Because $p=\vee _ia_i,$ Lemma \ref{22500} implies that $s_0^*s_0=p.$  
\end{proof}

Suppose that $A$ (resp. $B$) is the subalgebra of compact operators of the algebra $\Alg{\cl N}$ (resp. $\Alg{\cl M}$ ). It is well-known that 
$\Alg{\cl N}=A^{**}, \Alg{\cl M}=B^{**}.$  We define a map $\rho : B\rightarrow A$ such that  
$$\rho (k)=s_0^*ks_0,\;\;\forall \; k\in B.$$ Because $s_0s_0^*=I_{H_2}$, this map is a
homomorphism. If $k\in A$,  then 
$$pkp=s_0^*s_0ks_0^*s_0=\rho (s_0ks_0^*).$$ Thus $\rho (B)=pAp.$ Because $p\in \Delta (A^{**})^\prime =Z(\Delta (A^{**}))$, 
we have that $B\subset _{\sigma \Delta }A.$ 

In the following, we additionally assume that the dimensions of the atoms of $\cl N$ and $\cl M$ are one and that $\Delta (A^{**}), \;\Delta (B^{**})$
 are  maximal abelian self-adjoint algebras (MASAs). Such nests exist, see, for instance, Example 13.15 in \cite{dav}. We denote the 
algebras 
$$\hat B=B\oplus A\oplus A\oplus \ldots, \;\;\hat A=A\oplus A\oplus \ldots.$$
Because $B\subset _{\sigma \Delta }A$, we have that $\hat B\subset_{\sigma \Delta }\hat A.$
Furthermore, $$\hat A\cong (0\oplus \bb C I_{H_1} \oplus \bb C I_{H_1} \oplus \ldots )  \hat B 
(0\oplus \bb C I_{H_1} \oplus \bb C I_{H_1} \oplus \ldots )  .$$ Thus,  $\hat A\subset_{\sigma \Delta }\hat B.$
If $\subset _{\sigma \Delta }$ was a partial-order relation for non-self-adjoint algebras, then up to 
stable isomorphism we should have that $\hat A\sim _{\sigma \Delta }\hat B.$ 
Thus, the algebras $$\Omega =B^{**}\oplus A^{**}\oplus A^{**}\oplus \ldots, \;\;\;\Xi =A^{**}\oplus A^{**}\oplus \ldots$$
would be weakly stably isomorphic. Because $\Omega$ and $ \Xi $ are CSL algebras 
(see the definition of a CSL algebra in \cite{dav}), it follows from Theorem 3.2 in \cite{eleref} and 
Theorem 3.3 in \cite{eletro} that there would exists a $*$-isomorphism 
$$\theta : \Delta( \Omega)^\prime \rightarrow  \Delta (\Xi)^\prime  $$ such that 
$\theta (\Lat{\Omega })=\Lat{\Xi }.$ However, $\Delta( \Omega)$ and $\Delta (\Xi)$ are MASAs, and thus there exists a unitary $u$ 
such that $$\theta (x)=u^*xu, \;\;\;\forall \;x\;\in \Delta (\Omega ).$$ Therefore, $$u^*\Omega u=\Xi .$$
There exist completely contractive homomorphisms $\rho _k: B^{**}\rightarrow A^{**}, \;\;k=1,2,...$ 
such that 
$$u^*(x\oplus 0\oplus ...)u= \rho_1(x)\oplus  \rho _2(x)\oplus ...,\;\;\;\forall \; x\; \in B^{**}.$$ Suppose that 
$$u^*(   I_{H_2}\oplus 0\oplus ...  )u=p_1\oplus p_2\oplus ...$$ 
Because $$0\oplus ...\oplus 0\oplus p_i\oplus 0\oplus ...\leq p_1\oplus p_2\oplus ...$$ for all $i,$ 
we have that $$u(0\oplus ...0\oplus p_i\oplus 0\oplus ...)u^*\leq I_{H_2}\oplus 0\oplus ...  $$
Thus,  $$u  (0\oplus ...0\oplus p_i\oplus 0\oplus ...)  u^*=\hat p_i \oplus 0\oplus ...  $$
for orthogonal projections $\hat p_i\in \Delta (B^{**}), i\in \bb N.$ Observe that $\hat p_i\hat p_j=0$ for 
$i\neq j.$ If $x\in B^{**},$ then
$$u^*(x\oplus 0\oplus ...)u u^*(I_{H_2}\oplus 0...)u =u^*(I_{H_2}\oplus 0...)u u^*(x\oplus 0\oplus ...)u.$$
Thus, $$  (\rho_1(x)\oplus  \rho _2(x)\oplus ...)  (p_1\oplus p_2\oplus ...)=(p_1\oplus p_2\oplus ...)(\rho_1(x)\oplus  \rho _2(x)\oplus ...).
$$ We conclude that 
$$\rho _i(x)p_i=p_i\rho _i(x),\;\;\forall i\in \bb N, \;\;x\in B^{**}.$$
Thus, for all $x\in B^{**},$ we have that 
\begin{align*}  & u^*(x\oplus 0...)u u^*(\hat p_i\oplus 0\oplus ...)u =  (\rho_1(x)\oplus  \rho _2(x)\oplus ...)
 (0\oplus ...0\oplus p_i\oplus 0\oplus ...)=\\ & (0\oplus ...0\oplus p_i\oplus 0\oplus ...)(\rho_1(x)\oplus  \rho _2(x)\oplus ...)
=   u^*(\hat p_i\oplus 0\oplus ...)uu^*(x\oplus 0...)u   . 
\end{align*} 
Therefore, $\hat p_i$ is in the center of $B^{**}.$ However, as a nest algebra, $B^{**}$ has a trivial center. We can, therefore, conclude that 
there exists $i$ such that $$\hat p_i=I_{B^{**}}$$ and $\hat p_j=0$ for all $j\neq i.$ 
We obtain that 
\begin{equation} \label{xx}u^*(B^{**}\oplus 0\oplus ...)u=(0\oplus ...0\oplus p_iAp_i\oplus 0\oplus ...)
\end {equation}
By the same arguments, for the nest algebra $A^{**}$ there exists exactly one of 
the algebras 
$q_1B^{**}q_1, q_2A^{**}q_2, q_3A^{**}q_3,...$ with $q_1\in B^{**}, q_k\in A^{**}, k\geq 2$
such that 
$$u(0\oplus ...0\oplus A^{**}\oplus 0\oplus ...)u=(q_1B^{**}q_1 \oplus 0\oplus 0\oplus ...)$$
or 
$$u(0\oplus ...0\oplus A^{**}\oplus 0\oplus ...)u=(0\oplus  ...\oplus 0\oplus q_jA^{**}q_j \oplus 0\oplus ...).$$ 
 Here, in the left-hand side, $A^{**}$ is in the $i$-th position. 
The equality (\ref{xx}) implies that 
$$u(0\oplus ...0\oplus A^{**}\oplus 0\oplus ...)u=(q_1B^{**}q_1 \oplus 0\oplus 0\oplus ...)$$
We have proven that 
$$u^*(B^{**}\oplus 0\oplus ...)u\subseteq (0\oplus ...\oplus 0\oplus A^{**}\oplus 0\oplus ...)$$
and $$u(0\oplus ...\oplus 0\oplus A^{**}\oplus 0\oplus ...)u^*\subseteq (B^{**}\oplus 0\oplus ...)$$
We conclude that $$u^*(B^{**}\oplus 0\oplus ...)u=(0\oplus ...\oplus 0\oplus A^{**}\oplus 0\oplus ...)$$
Thus, the nest algebras $A^{**}$ and $B^{**}$ are completely isometrically isomorphic. It follows that their diagonals 
$\Delta (A^{**})$ and $\Delta (B^{**})$ are $*$-isomorphic. However, $\Delta (B^{**})$ 
is an atomic MASA, and $\Delta (A^{**})$ is a MASA with a nontrivial continuous part. This contradiction shows that $\hat A$ and $\hat B$ 
are not $\sigma$-strongly $ \Delta $-equivalent.

\end{document}